\documentclass[reqno, 12pt]{amsart}
\usepackage{enumerate}
\usepackage{amssymb,url,color}
\usepackage{hyperref}
\usepackage{amsmath,amsthm}
\usepackage{amsfonts}
\usepackage{mathrsfs}
\usepackage{stmaryrd}
\usepackage{graphicx}
\usepackage{bbm}

\allowdisplaybreaks[4]

\newtheorem{thm}{Theorem}[section]

\newtheorem{lem}{Lemma}[section]
\newtheorem{rem}{Remark}[section]
\newtheorem{exa}{Example}[section]
\theoremstyle{Problem}

\theoremstyle{definition}

\numberwithin{equation}{section}
\newcommand{\pp}{\mathbb{P}}
\newcommand{\ee}{\mathbb{E}}

\newcommand{\xx}{\mathcal{X}}

\newcommand{\dd}{\mathfrak{D}}
\newcommand{\rr}{\mathbb{R}}

\def\beq{\begin{equation}}
\def\deq{\end{equation}}
\allowdisplaybreaks 

\bfseries\rmfamily 

\begin{document}

\title[plug-in estimator of entropy]
{Some asymptotic behaviors for the plug-in estimator of entropy}
\thanks{This work is supported by National Natural Science Foundation of China (NSFC-11971154).}

\author[Y. Miao]{Yu Miao}
\address[Y. Miao]{College of Mathematics and Information Science, Henan Normal University, Henan Province, 453007, China; Henan Engineering Laboratory for Big Data Statistical Analysis and Optimal Control, Henan Normal University, Henan Province, 453007, China.} \email{\href{mailto: Y. Miao
<yumiao728@gmail.com>}{yumiao728@gmail.com}; \href{mailto: Y. Miao <yumiao728@126.com>}{yumiao728@126.com}}

\author[Z. H. Yu]{Zhenhong Yu}
\address[Z. H. Yu]{College of Mathematics and Information Science, Henan Normal University, Henan Province, 453007, China.} \email{\href{mailto: Z. H. Yu
<zhenhongyu2022@126.com>}{zhenhongyu2022@126.com}}

\begin{abstract}
In the present paper, we consider the plug-in estimator of Shannon's entropy defined on a finite alphabet which is assumed to dynamically vary as the sample size increases. The asymptotic behaviors for the plug-in estimator, such as, asymptotic normality, Berry-Esseen bound and moderate deviation principle, are established.
\end{abstract}

\maketitle

\section{Introduction}
Suppose that $X$ is a discrete random variable with an unknown distribution $\{p(i), i\in\xx\}$ on an alphabet $\xx=\{i, 1\le i\le K\}$,
where $K$ denotes either a finite integer or $\infty$ and
$$
p(i)=\pp(X=i),\ \ \ 1\le i\le K.
$$
Shannon \cite{Shannon} introduced the following entropy
$$
 H=\ee(-\ln p(X))=-\sum_{i=1}^K p(i)\ln p(i),
$$
which is often refereed to as Shannon's entropy. Let $\{X_n, n\ge 1\}$ be a sequence of independent identically distributed random variables with common distribution $\{p(i), i\in\xx\}$ on $\xx$. Define the plug-in estimator of the entropy $H$ as
\beq\label{H}
\hat H_n=-\sum_{i=1}^K \hat p_n(i)\ln \hat p_n(i)
\deq
where
$$
\hat p_n(i)=\frac{1}{n}\sum_{j=1}^nI_{\{X_j=i\}}
$$
is the empirical distribution induced by the samples $(X_1, X_2, \cdots, X_n)$ on $\xx$.

For the case that $K$ is fixed and finite, let $\sigma^2=Var(\ln p(X_1))>0$, then Basharin \cite{Basharin} gave the following central limit theorem for $\hat H_n$
$$
\frac{\sqrt{n}}{\sigma}(\hat H_n-H)\xrightarrow{\dd} N(0,1)
$$
where $N(0, 1)$ denotes the standard normal random variable. For the case that $K=K(n)$ varies as the sample size $n$ increase,
let $\{X_{k,n}, 1\le k\le K(n), n\ge 1\}$ be an array of independent identically distributed random variables with common distribution $\{p_n(i),  1\le i\le K(n), n\ge 1\}$, i.e., for any $n\ge 1$,
$$
\pp(X_{1,n}=i)=p_n(i),  \ \ \  \ 1\le i\le K(n).
$$
Assume that the sequence $K(n)$ satisfies
\beq\label{0}
K(n)\to\infty,\ \ K(n)=o(\sqrt{n})\ \ \text{and}\ \ \liminf_{n\to\infty}n^{1-\alpha}\sigma^2_n>0
\deq
for some $\alpha>0$, where
$$
 \sigma_n^2=Var(\ln p_n(X_{1,n}))=\sum_{i=1}^{K(n)}p_{n}(i)\ln^2 p_{n}(i)-\left(\sum_{i=1}^{K(n)}p_{n}(i)\ln p_{n}(i)\right)^2,
$$
then Paninski \cite{Paninski} proved
$$
\frac{\sqrt{n}}{\sigma_n}(\hat H_n-H_n)\xrightarrow{\dd} N(0,1),
$$
where $H_n$ denotes Shannon's entropy with distribution $\{p_n(i), 1\le i\le K(n), n\ge 1\}$. For the case $K=\infty$, Antos and Kontoyiannis \cite{AnKo} studied the rates of convergence for $\hat H_n$ under a variety of tail conditions on $\{p(i), i\ge 1\}$ and showed that no universal rate of convergence exists for any sequence of estimators.
If $\{X_n, n\ge 1\}$ is a sequence of independent identically distributed random variables with common nonuniform distribution $\{p(i), i\ge 1\}$ satisfying
$\ee(\ln p(X_1))^2<\infty$. Assume that there exists an integer valued function $K(n)$  such that
$$
K(n)\to\infty,\ \ K(n)=o(\sqrt{n})\ \ \text{and}\ \ \sqrt{n}\sum_{i=K(n)}^\infty p(i)\ln p(i)\to 0,
$$
then Zhang and Zhang \cite{ZhZh12} obtained
$$
\frac{\sqrt{n}}{\sigma}(\hat H_n-H)\xrightarrow{\dd} N(0,1)
$$
where $\sigma^2=Var(\ln p(X_1))$.

Based on the above works, in the present paper, we shall continue to study the asymptotic behaviors for the plug-in estimator of entropy defined on a finite alphabet which is assumed to dynamically vary as the sample size increases, and establish the asymptotic normality, Berry-Esseen bound and moderate deviation principle of the plug-in estimator of entropy. Throughout the following sections, let the symbol $C$ represent positive constants whose values may change from one place to another.

\section{Main results}

In this section, we consider the case $K=K(n)$, which is assumed to dynamically vary as the sample size $n$ increases, and define
$$
 \sigma_n^2=Var(\ln p_n(X_{1,n}))=\sum_{i=1}^{K(n)}p_{n}(i)\ln^2 p_{n}(i)-\left(\sum_{i=1}^{K(n)}p_{n}(i)\ln p_{n}(i)\right)^2.
$$
Firstly, we study the asymptotic normality of the plug-in estimator.

\begin{thm}\label{thm2-1} Let $\{X_{k,n}, 1\le k\le K(n), n\ge 1\}$ be an array of independent identically distributed random variables with common nonuniform distribution $\{p_n(i),  1\le i\le K(n), n\ge 1\}$, i.e., for any $n\ge 1$,
$$
\pp(X_{1,n}=i)=p_n(i),  \ \ \  \ 1\le i\le K(n).
$$
Assume that the sequence $K(n)$ satisfies
\beq\label{1}
K(n)\to\infty,\ \ K(n)=o(\sqrt{n}\sigma_n)\ \  \text{and}\ \ \frac{\sqrt{n}\sigma_n}{\ln n}\to \infty,
\deq
then we have
$$
\frac{\sqrt{n}}{\sigma_n}(\hat H_n-H_n)\xrightarrow{\dd} N(0,1),
$$
where $H_n$ denotes Shannon's entropy with distribution $\{p_n(i), 1\le i\le K(n), n\ge 1\}$.
\end{thm}
\begin{rem}
If $\sqrt{n}\sigma_n=\ln^{t} n$ for some $t>1$, then we can not choose $\alpha>0$ such that the condition (\ref{0}) holds, i.e., for any $\alpha>0$, it follows that
$$
\lim_{n\to\infty}n^{1-\alpha}\sigma^2_n=0.
$$
However, it is easy to check that the condition (\ref{1}) holds, i.e.,
$$
\frac{\sqrt{n}\sigma_n}{\ln n}= \ln^{t-1} n \to \infty.
$$ Hence the condition (\ref{1}) is weaker than the condition (\ref{0}).
\end{rem}
\begin{thm}\label{thm2-2}
 Let $\{X_{k,n}, 1\le k\le K(n), n\ge 1\}$ be an array of independent identically distributed random variables with common nonuniform distribution $\{p_n(i),  1\le i\le K(n), n\ge 1\}$, i.e., for any $n\ge 1$,
$$
\pp(X_{1,n}=i)=p_n(i),  \ \ \  \ 1\le i\le K(n)
$$
where the sequence $K(n)$ satisfies the condition (\ref{1}).
Then for any $0\le \delta\le 1$, we have
\beq\label{3}
\sup_x\left|\pp\left(\frac{\sqrt{n}}{\sigma_n}(\hat H_n-H_n)\le x\right)-\Phi(x)\right|\le C\left(\frac{\ee|T_{1,n}|^{2+\delta}}{n^{\delta/2}\sigma_n^{2+\delta}}+\sqrt{\frac{K(n)}{\sqrt{n}\sigma_n}}\right),
\deq
where $\Phi$ denotes the standard normal distribution function, $C$ is a positive constant and
$$
\ee|T_{1,n}|^{2+\delta}=\ee\left|\sum_{i=1}^{K(n)}(I_{\{X_{1, n}=i\}}-p_n(i))\ln p_n(i)\right|^{2+\delta}.
$$
\end{thm}
\begin{rem}\label{rem-2}
By using the elementary inequality $(a+b)^{2+\delta}\le C(a^{2+\delta}+b^{2+\delta})$, we have
$$
\aligned
\ee|T_{1,n}|^{2+\delta}=&\ee\left|\sum_{i=1}^{K(n)}(I_{\{X_{1, n}=i\}}-p_n(i))\ln p_n(i)\right|^{2+\delta}\\
\le &C\ee\left|\sum_{i=1}^{K(n)}I_{\{X_{1, n}=i\}}\ln p_n(i)\right|^{2+\delta}+C\left|-\sum_{i=1}^{K(n)}p_n(i)\ln p_n(i)\right|^{2+\delta}\\
= &C\sum_{i=1}^{K(n)}\ee\left[\left|\ln p_n(i)\right|^{2+\delta}I_{\{X_{1, n}=i\}}\right]+ C\left|H_n\right|^{2+\delta}\\
= &C\sum_{i=1}^{K(n)} p_n(i)\left|\ln p_n(i)\right|^{2+\delta}+ C\left|H_n\right|^{2+\delta}.
\endaligned
$$
\end{rem}

\begin{thm}\label{thm2-3}
 Let $\{X_{k,n}, 1\le k\le K(n), n\ge 1\}$ be an array of independent identically distributed random variables with common nonuniform distribution $\{p_n(i),  1\le i\le K(n), n\ge 1\}$, i.e., for any $n\ge 1$,
$$
\pp(X_{1,n}=i)=p_n(i),  \ \ \  \ 1\le i\le K(n).
$$
Assume that there exists a positive constant $\delta>0$ such that
\beq\label{4}
\sup_{n\ge 1}\ee\exp\left(\frac{\delta}{\sigma_n}\left|\sum_{i=1}^{K(n)}(I_{\{X_{1, n}=i\}}-p_n(i))\ln p_n(i)\right|\right)<\infty.
\deq
Then for any $r > 0$, we have
\beq
\lim_{n\rightarrow\infty}\frac{1}{b_n^2}\log\pp
\left(\frac{\sqrt{n}}{b_n\sigma_n}|\hat{H}_n-H|>r\right)= -\frac{r^2}{2}
\deq
where the moderate deviation scale $\{b_n, n\ge 1\}$ is a sequence of positive numbers satisfying
$$
b_n\rightarrow\infty,\ \ \frac{b_n}{\sqrt{n}}\rightarrow 0
$$
and
$$
\lim_{n\rightarrow\infty}\frac{1}{b_n^2}\log \left[\sum_{i=1}^{K(n)}\exp\left(-2\varepsilon\sqrt{n}b_n\sigma_np_n^2(i)\right)\right]=-\infty.
$$
\end{thm}

\begin{rem}\label{rem-3}
In this remark, we shall give the bound of the exponential moment in (\ref{4}). Since
$$
\aligned
&\ee\exp\left(\frac{\delta}{\sigma_n}\left|\sum_{i=1}^{K(n)}I_{\{X_{1, n}=i\}}\ln p_n(i)\right|\right)\\
=&\sum_{i=1}^{K(n)}\ee\left[\exp\left(\frac{\delta}{\sigma_n}\left|\sum_{i=1}^{K(n)}I_{\{X_{1, n}=i\}}\ln p_n(i)\right|\right)I_{\{X_{1, n}=i\}}\right]\\
=&\sum_{i=1}^{K(n)}\exp\left(\frac{\delta}{\sigma_n}\left|\ln p_n(i)\right|\right)p_n(i)
=\sum_{i=1}^{K(n)}(p_n(i))^{1-\frac{\delta}{\sigma_n}},
\endaligned
$$
then we have
$$
\aligned
&\sup_{n\ge 1}\ee\exp\left(\frac{\delta}{\sigma_n}\left|\sum_{i=1}^{K(n)}(I_{\{X_{1, n}=i\}}-p_n(i))\ln p_n(i)\right|\right)\\
\le &\sup_{n\ge 1}\ee\exp\left(\frac{\delta}{\sigma_n}\left|\sum_{i=1}^{K(n)}I_{\{X_{1, n}=i\}}\ln p_n(i)\right|\right)\exp\left(\frac{\delta}{\sigma_n}\left|\sum_{i=1}^{K(n)}p_n(i)\ln p_n(i)\right|\right)\\
\le & \sup_{n\ge 1}\left[\sum_{i=1}^{K(n)}(p_n(i))^{1-\frac{\delta}{\sigma_n}}\exp\left(\frac{\delta}{\sigma_n}H_n\right)\right].
\endaligned
$$
\end{rem}

\section{Examples}
In this section, we give several examples to show that the results in the section 2 hold.
\begin{exa}
For every $i=1,2, \cdots, K(n)$, let $p_n(i)=(C_ni)^{-1}$,
where
$$
C_n=\sum_{i=1}^{K(n)}\frac{1}{i}\sim \ln K(n),
$$
then we have
\beq\label{3.2}
\aligned
 \sigma_n^2=&\sum_{i=1}^{K(n)}p_{n}(i)\ln^2 p_{n}(i)-\left(\sum_{i=1}^{K(n)}p_{n}(i)\ln p_{n}(i)\right)^2\\
 =& \sum_{i=1}^{K(n)}\frac{1}{C_ni}\left(\ln C_{n}+\ln i\right)^2-\left(\sum_{i=1}^{K(n)}\frac{1}{C_ni}\left(\ln C_n+\ln i\right)\right)^2\\
 =& (\ln C_{n})^2 +2\frac{\ln C_n}{C_n}\sum_{i=1}^{K(n)}\frac{\ln i}{i}+ \frac{1}{C_n}\sum_{i=1}^{K(n)}\frac{(\ln i)^2}{i}-\left(\ln C_n+ \frac{1}{C_n}\sum_{i=1}^{K(n)}\frac{\ln i}{i}\right)^2\\
=&\frac{1}{C_n}\sum_{i=1}^{K(n)}\frac{(\ln i)^2}{i}-\frac{1}{C_n^2}\left(\sum_{i=1}^{K(n)}\frac{\ln i}{i}\right)^2  \\
\sim&  \frac{1}{3}(\ln K(n))^2-\frac{1}{4}(\ln K(n))^2=\frac{1}{12}(\ln K(n))^2.
\endaligned
\deq

\vskip 5pt

{\rm Conditions in Theorem \ref{thm2-1}}: We can take $K(n)\to \infty$ such that $K(n)=o(\sqrt{n}\ln K(n))$. For the sequence $K(n)$, it is easy to check
$$
\frac{\sqrt{n}\ln K(n)}{\ln n}\to \infty.
$$
Hence the condition (\ref{1}) holds.

\vskip 5pt

{\rm Conditions in Theorem \ref{thm2-2}}: We can check that
$$
\aligned
H_n=-\sum_{i=1}^{K(n)}p_{n}(i)\ln p_{n}(i)=&\sum_{i=1}^{K(n)}\frac{1}{C_ni}(\ln C_n+\ln i)\\
\sim&\ln C_n+\frac{1}{2}\ln K(n)\sim \frac{1}{2}\ln K(n)
\endaligned
$$
and for any $0\le \delta\le 1$,
$$
\aligned
\sum_{i=1}^{K(n)} p_n(i)\left|\ln p_n(i)\right|^{2+\delta}=&\sum_{i=1}^{K(n)}\frac{1}{C_ni}(\ln C_n+\ln i)^{2+\delta}\\
\le & C\sum_{i=1}^{K(n)}\frac{1}{C_ni}(\ln C_n)^{2+\delta}+C\sum_{i=1}^{K(n)}\frac{1}{C_ni}(\ln i)^{2+\delta}\\
\sim& C(\ln C_n)^{2+\delta}+\frac{C}{3+\delta}(\ln K(n))^{2+\delta}\\
\sim& \frac{C}{3+\delta}(\ln K(n))^{2+\delta}.
\endaligned
$$
From Remark \ref{rem-2}, we have
$$
\aligned
\ee|T_{1,n}|^{2+\delta}\le &C\sum_{i=1}^{K(n)} p_n(i)\left|\ln p_n(i)\right|^{2+\delta}+ C\left|H_n\right|^{2+\delta}\le C(\ln K(n))^{2+\delta}.
\endaligned
$$
Hence, we have
\beq
\sup_x\left|\pp\left(\frac{\sqrt{n}}{\sigma_n}(\hat H_n-H_n)\le x\right)-\Phi(x)\right|\le C\left(\frac{1}{n^{\delta/2}}+\sqrt{\frac{K(n)}{\sqrt{n}\ln K(n)}}\right)=:\Delta_n.
\deq
For the case $0<\delta<\frac{1}{2}$, if the sequence $K(n)$ satisfies
$$
\frac{K(n)}{\ln K(n)}< n^{\frac{1}{2}-\delta},
$$
then $\Delta_n\le Cn^{-\delta/2}$; if the sequence $K(n)$ satisfies
$$
n^{\frac{1}{2}-\delta}\le \frac{K(n)}{\ln K(n)}<\sqrt{n},
$$
then
$$
\Delta_n\le C\sqrt{\frac{K(n)}{\sqrt{n}\ln K(n)}}.
$$
For the case $\frac{1}{2}<\delta<1$, we can take any sequence $K(n)$ such that $K(n)=o(\sqrt{n}\ln K(n))$, then
$$
\Delta_n\le C\sqrt{\frac{K(n)}{\sqrt{n}\ln K(n)}}.
$$

\vskip 5pt

{\rm Conditions in Theorem \ref{thm2-3}}: We can check that
$$
\aligned
\sum_{i=1}^{K(n)}(p_n(i))^{1-\frac{\delta}{\sigma_n}}
= &
\left(\frac{1}{C_n}\right)^{1-\frac{\delta}{\sigma_n}}\sum_{i=1}^{K(n)}\left(\frac{1}{i}\right)^{1-\frac{\delta}{\sigma_n}}\\
\le & C\frac{\sigma_n}{\delta}\left(\frac{1}{C_n}\right)^{1-\frac{\delta}{\sigma_n}}(K(n))^{\frac{\delta}{\sigma_n}} \\
\le & C\left(K(n)\ln K(n)\right)^{\frac{\sqrt{12}\delta}{\ln K(n)}}\\
=&C\exp\left(\frac{\sqrt{12}\delta}{\ln K(n)}\left[\ln K(n)+\ln\ln K(n)\right]\right),
\endaligned
$$
which implies that
$$
\aligned
& \sup_{n\ge 1}\left[\sum_{i=1}^{K(n)}(p_n(i))^{1-\frac{\delta}{\sigma_n}}\exp\left(\frac{\delta}{\sigma_n}H_n\right)\right]\\
\le &C\sup_{n\ge 1}\exp\left(\frac{\sqrt{12}\delta}{\ln K(n)}\left[\ln K(n)+\ln\ln K(n)\right]+\sqrt{3}\delta\right)<\infty.
\endaligned
$$
Furthermore, we have
$$
\aligned
\sum_{i=1}^{K(n)}\exp\left(-2\varepsilon\sqrt{n}b_n\sigma_np_n^2(i)\right)
= &\sum_{i=1}^{K(n)}\exp\left(-\frac{2\varepsilon\sqrt{n}b_n}{\sqrt{12} i^2\ln K(n)}\right)\\
\le& K(n)\exp\left(-\frac{2\varepsilon\sqrt{n}b_n}{\sqrt{12} (K(n))^2\ln K(n)}\right).
\endaligned
$$
By taking the sequence $b_n$ such that
$$
\frac{1}{b_n^2}\log K(n)-\frac{2\varepsilon\sqrt{n}}{\sqrt{12}b_n (K(n))^2\ln K(n)}\to -\infty,
$$
then we can get
$$
\lim_{n\to\infty}\frac{1}{b_n^2}\log \left[\sum_{i=1}^{K(n)}\exp\left(-2\varepsilon\sqrt{n}b_n\sigma_np_n^2(i)\right)\right]=-\infty.
$$
For example, we take $K(n)=n^{\lambda}$ for some $0<\lambda<1/4$ and $b_n=n^{r}$ for $0<r<\frac{1}{2}-2\lambda$.
\end{exa}

\begin{exa}
For every $i=1,2, \cdots, K(n)$, let $p_n(i)=(C_ne^i)^{-1}$
where
$$
C_n=\sum_{i=1}^{K(n)}\frac{1}{e^i}=\frac{e^{-1}(1-e^{-K(n)})}{1-e^{-1}},
$$
then we have
\beq\label{3.2-1}
\aligned
 \sigma_n^2=&\sum_{i=1}^{K(n)}p_{n}(i)\ln^2 p_{n}(i)-\left(\sum_{i=1}^{K(n)}p_{n}(i)\ln p_{n}(i)\right)^2\\
 =& \sum_{i=1}^{K(n)}\frac{1}{C_ne^i}\left(\ln C_{n}+ i\right)^2-\left(\sum_{i=1}^{K(n)}\frac{1}{C_ne^i}\left(\ln C_n+ i\right)\right)^2\\
 =& (\ln C_{n})^2 +2\frac{\ln C_n}{C_n}\sum_{i=1}^{K(n)}\frac{ i}{e^i}+ \frac{1}{C_n}\sum_{i=1}^{K(n)}\frac{i^2}{e^i}-\left(\ln C_n+ \frac{1}{C_n}\sum_{i=1}^{K(n)}\frac{ i}{e^i}\right)^2\\
=&\frac{1}{C_n}\sum_{i=1}^{K(n)}\frac{ i^2}{e^i}-\frac{1}{C_n^2}\left(\sum_{i=1}^{K(n)}\frac{ i}{e^i}\right)^2.
\endaligned
\deq
Obviously, there is a positive constant $C$, such that $\sigma_n^2<C$. Next, we discuss the lower bound of $\sigma_n^2$.
For any positive integer $M\ge 2$, we have
$$
\aligned
\sum_{i=1}^{K(n)}\frac{i}{e^i}\le&\sum_{i=1}^{M}\frac{i}{e^i}+\int_M^{K(n)}xe^{-x}dx\\
=&\sum_{i=1}^{M}\frac{i}{e^i}+(M+1)e^{-M}-(K(n)+1)e^{-K(n)}
\endaligned
$$
and
$$
\aligned
\sum_{i=1}^{K(n)}\frac{i^2}{e^i}\ge&\sum_{i=1}^{M}\frac{i^2}{e^i}+\int_{M+1}^{K(n)+1}x^2e^{-x}dx\\
=&\sum_{i=1}^{M}\frac{i^2}{e^i}+[(M+1)^2+2(M+2)]e^{-(M+1)}\\
&\ \ \ \ \ \ \ \ \ \ \ \ \ \ \ -[(K(n)+1)^2+2(K(n)+2)]e^{-(K(n)+1)}.
\endaligned
$$
For any small $\varepsilon>0$ and for all $n$ large enough, we have
$$
\aligned
\frac{1}{C_n^2}\left(\sum_{i=1}^{K(n)}\frac{i}{e^i}\right)^2\le\frac{(e-1)^2(1+\varepsilon)^2}{(1-\varepsilon)^2}\left(\sum_{i=1}^{2}\frac{i}{e^i}+3e^{-2}\right)^2
\endaligned
$$
and
$$
\aligned
\frac{1}{C_n}\sum_{i=1}^{K(n)}\frac{i^2}{e^i}\ge (e-1)(1-\varepsilon)\left(\sum_{i=1}^{5}\frac{i^2}{e^i}+50e^{-6}\right).
\endaligned
$$
It is easy to check that
$$
\left(\sum_{i=1}^{5}\frac{i^2}{e^i}+50e^{-6}\right)>(e-1)\left(\sum_{i=1}^{2}\frac{i}{e^i}+3e^{-2}\right)^2,
$$
which implies that for any $\varepsilon>0$ small enough, we have
$$
(1-\varepsilon)\left(\sum_{i=1}^{5}\frac{i^2}{e^i}+50e^{-6}\right)>
\frac{(e-1)(1+\varepsilon)^2}{(1-\varepsilon)^2}\left(\sum_{i=1}^{2}\frac{i}{e^i}+3e^{-2}\right)^2.
$$
Hence, there exists a positive constant $c$ such that $\sigma_n^2>c$ for all $n$.
\vskip 5pt

{\rm Conditions in Theorem \ref{thm2-1}}: We can take $K(n)\to \infty$ such that $K(n)=o(\sqrt{n})$. For the sequence $K(n)$, it is easy to check
$$
\frac{\sqrt{n}}{\ln n}\to \infty.
$$
Hence the condition (\ref{1}) holds.

\vskip 5pt

{\rm Conditions in Theorem \ref{thm2-2}}: We can check that
$$
\aligned
H_n=-\sum_{i=1}^{K(n)}p_{n}(i)\ln p_{n}(i)=&\sum_{i=1}^{K(n)}\frac{1}{C_ne^i}(\ln C_n+i)<C
\endaligned
$$
and for any $0\le \delta\le 1$,
$$
\aligned
\sum_{i=1}^{K(n)} p_n(i)\left|\ln p_n(i)\right|^{2+\delta}=\sum_{i=1}^{K(n)}\frac{1}{C_ne^i}(\ln C_n+ i)^{2+\delta}
\le C\sum_{i=1}^{K(n)}\frac{i^{2+\delta}}{C_ne^i}<C.
\endaligned
$$
From Remark \ref{rem-2}, we have
$$
\aligned
\ee|T_{1,n}|^{2+\delta}\le &C\sum_{i=1}^{K(n)} p_n(i)\left|\ln p_n(i)\right|^{2+\delta}+ C\left|H_n\right|^{2+\delta}\le C.
\endaligned
$$
Hence, we have
\beq
\sup_x\left|\pp\left(\frac{\sqrt{n}}{\sigma_n}(\hat H_n-H_n)\le x\right)-\Phi(x)\right|\le C\left(\frac{1}{n^{\delta/2}}+\sqrt{\frac{K(n)}{\sqrt{n}}}\right)=:\Delta_n.
\deq
For the case $0<\delta<\frac{1}{2}$, if the sequence $K(n)$ satisfies
$$
 K(n)<  n^{\frac{1}{2}-\delta},
$$
then $\Delta_n\le Cn^{-\delta/2}$; if the sequence $K(n)$ satisfies
$$
n^{\frac{1}{2}-\delta}\le K(n)<\sqrt{n},
$$
then
$$
\Delta_n\le C\sqrt{\frac{K(n)}{\sqrt{n}}}.
$$
For the case $\frac{1}{2}<\delta<1$, we can take any sequence $K(n)$ such that $K(n)=o(\sqrt{n})$, then
$$
\Delta_n\le C\sqrt{\frac{K(n)}{\sqrt{n}}}.
$$

\vskip 5pt

{\rm Conditions in Theorem \ref{thm2-3}}: We can take $\delta>0$ small enough such that $\delta/\sigma_n<1$, then it follows that
$$
\aligned
\sum_{i=1}^{K(n)}(p_n(i))^{1-\frac{\delta}{\sigma_n}}
=
\left(\frac{1}{C_n}\right)^{1-\frac{\delta}{\sigma_n}}\sum_{i=1}^{K(n)}\left(\frac{1}{e^i}\right)^{1-\frac{\delta}{\sigma_n}}
\le  C,
\endaligned
$$
which implies that
$$
\aligned
 \sup_{n\ge 1}\left[\sum_{i=1}^{K(n)}(p_n(i))^{1-\frac{\delta}{\sigma_n}}\exp\left(\frac{\delta}{\sigma_n}H_n\right)\right]
<\infty.
\endaligned
$$
Furthermore, we have
$$
\aligned
\sum_{i=1}^{K(n)}\exp\left(-2\varepsilon\sqrt{n}b_n\sigma_np_n^2(i)\right)
\le \sum_{i=1}^{K(n)}\exp\left(-\frac{2\varepsilon c\sqrt{n}b_n}{C_n e^i}\right)
\le  K(n)\exp\left(-\frac{2\varepsilon C\sqrt{n}b_n}{e^{K(n)}}\right).
\endaligned
$$
By taking the sequence $b_n$ such that
$$
\frac{1}{b_n^2}\log K(n)-\frac{2\varepsilon\sqrt{n}}{b_n e^{K(n)}}\to -\infty,
$$
then we can get
$$
\lim_{n\to\infty}\frac{1}{b_n^2}\log \left[\sum_{i=1}^{K(n)}\exp\left(-2\varepsilon\sqrt{n}b_n\sigma_np_n^2(i)\right)\right]=-\infty.
$$
For example, we take $K(n)=(\ln n)^{\lambda}$ for some $0<\lambda<\frac{1}{2}$ and $b_n=n^{r}$ for $0<r<\frac{1}{2}-\lambda$.
\end{exa}

\begin{exa}
For every $i=2, \cdots, K(n)$ let $p_n(i)=(C_ni\ln i)^{-1}$,
where
$$
C_n=\sum_{i=2}^{K(n)}\frac{1}{i\ln i}\sim \ln\ln K(n),
$$
then we have
\beq\label{3.2}
\aligned
 \sigma_n^2=&\sum_{i=1}^{K(n)}p_{n}(i)\ln^2 p_{n}(i)-\left(\sum_{i=1}^{K(n)}p_{n}(i)\ln p_{n}(i)\right)^2\\
 =& \sum_{i=1}^{K(n)}\frac{1}{C_ni\ln i}\left(\ln C_{n}+\ln(i\ln i)\right)^2-\left(\sum_{i=1}^{K(n)}\frac{1}{C_ni\ln i}\left(\ln C_n+\ln(i\ln i)\right)\right)^2\\
 =& (\ln C_{n})^2 +2\frac{\ln C_n}{C_n}\sum_{i=1}^{K(n)}\frac{\ln(i\ln i)}{i\ln i}+ \frac{1}{C_n}\sum_{i=1}^{K(n)}\frac{(\ln (i\ln i))^2}{i\ln i}\\
 &\ \ \ \ \ \ \ \ \ \ \ \ \ \ \ \ \ -\left(\ln C_n+ \frac{1}{C_n}\sum_{i=1}^{K(n)}\frac{\ln (i\ln i)}{i\ln i}\right)^2\\
=&\frac{1}{C_n}\sum_{i=1}^{K(n)}\frac{(\ln (i\ln i))^2}{i\ln i}-\frac{1}{C_n^2}\left(\sum_{i=1}^{K(n)}\frac{\ln (i\ln i)}{i\ln i}\right)^2  \\
\sim&  \frac{1}{2}\frac{(\ln K(n))^2}{\ln\ln K(n)}-\frac{(\ln K(n))^2}{(\ln\ln K(n))^2}\sim \frac{1}{2}\frac{(\ln K(n))^2}{\ln\ln K(n)}.
\endaligned
\deq

\vskip 5pt

{\rm Conditions in Theorem \ref{thm2-1}}: We can take $K(n)\to \infty$ such that
$$
K(n)=o\left(\sqrt{n}\frac{\ln K(n)}{\sqrt{\ln\ln K(n)}}\right).
$$ For the sequence $K(n)$, it is easy to check
$$
\frac{\sqrt{n}\ln K(n)}{\ln n\sqrt{\ln\ln K(n)}}\to\infty.
$$
Hence the condition (\ref{1}) holds.

\vskip 5pt

{\rm Conditions in Theorem \ref{thm2-2}}: We can check that
$$
\aligned
H_n=-\sum_{i=1}^{K(n)}p_{n}(i)\ln p_{n}(i)=&\sum_{i=1}^{K(n)}\frac{1}{C_ni\ln i}(\ln C_n+\ln(i\ln i))\\
\sim&\ln C_n+\frac{\ln K(n)}{\ln\ln K(n)}\sim\frac{\ln K(n)}{\ln\ln K(n)}
\endaligned
$$
and for any $0\le \delta\le 1$,
$$
\aligned
\sum_{i=1}^{K(n)} p_n(i)\left|\ln p_n(i)\right|^{2+\delta}=&\sum_{i=1}^{K(n)}\frac{1}{C_ni\ln i}(\ln C_n+\ln(i\ln i))^{2+\delta}\\
\le & C\sum_{i=1}^{K(n)}\frac{(\ln C_n)^{2+\delta}}{C_ni\ln i}+C\sum_{i=1}^{K(n)}\frac{(\ln( i\ln i))^{2+\delta}}{C_ni\ln i}\\
\sim& C(\ln C_n)^{2+\delta}+\frac{C}{(2+\delta)C_n}(\ln K(n))^{2+\delta}\\
\sim& \frac{C}{(2+\delta)}\frac{(\ln K(n))^{2+\delta}}{\ln\ln K(n)}.
\endaligned
$$
From Remark \ref{rem-2}, we have
$$
\aligned
\ee|T_{1,n}|^{2+\delta}\le &C\sum_{i=1}^{K(n)} p_n(i)\left|\ln p_n(i)\right|^{2+\delta}+ C\left|H_n\right|^{2+\delta}\le C\frac{(\ln K(n))^{2+\delta}}{\ln\ln K(n)}.
\endaligned
$$
Hence, we have
\beq
\aligned
\sup_x&\left|\pp\left(\frac{\sqrt{n}}{\sigma_n}(\hat H_n-H_n)\le x\right)-\Phi(x)\right|\\
&\le C\left(\frac{(\ln\ln K(n))^{\delta/2}}{n^{\delta/2}}+\sqrt{\frac{K(n)\sqrt{\ln\ln K(n)}}{\sqrt{n}\ln K(n)}}\right)=:\Delta_n.
\endaligned
\deq
For the case $0<\delta<\frac{1}{2}$, if the sequence $K(n)$ satisfies
$$
\frac{K(n)}{\ln K(n)}< \left(\frac{n}{\ln\ln K(n)}\right)^{\frac{1}{2}-\delta},
$$
then
$$
\Delta_n\le C\frac{(\ln\ln K(n))^{\delta/2}}{n^{\delta/2}};
$$
 if the sequence $K(n)$ satisfies
$$
\left(\frac{n}{\ln\ln K(n)}\right)^{\frac{1}{2}-\delta}\le \frac{K(n)}{\ln K(n)}<\left(\frac{n}{\ln\ln K(n)}\right)^{\frac{1}{2}},
$$
then
$$
\Delta_n\le C\sqrt{\frac{K(n)}{\sqrt{n}\ln K(n)}}.
$$
For the case $\frac{1}{2}\le\delta<1$, we can take any sequence $K(n)$ such that
$$
K(n)=o\left(\frac{\sqrt{n}\ln K(n)}{\sqrt{\ln\ln K(n)}}\right),
$$ then
$$
\Delta_n\le C\sqrt{\frac{K(n)\sqrt{\ln\ln K(n)}}{\sqrt{n}\ln K(n)}}.
$$

\vskip 5pt

{\rm Here we need to remark that the Conditions in Theorem \ref{thm2-3} can not be satisfied for this example}.
\end{exa}

\section{Proofs of main results}
We state some useful lemmas to prove these main results.
\begin{lem}\label{lem1}\cite{C-S}
 Let $X_1, X_2, \cdots, X_n$ be independent and not necessarily identically distributed random variables with zero means and finite variances. Define $W=\sum_{k=1}^nX_k$ and assume that $Var(W)=1$. Let $F$ be the distribution function of $W$ and $\Phi$ the standard normal distribution function. Then there exists an absolute constant $C$ such that for every real number $x$,
 $$
 |F(x)-\Phi(x)|\le C\sum_{i=1}^n\left\{\frac{\ee X_i^2I(|X_i|>1+|x|)}{(1+|x|)^2}+
 \frac{\ee |X_i|^3I(|X_i|\le 1+|x|)}{(1+|x|)^3}\right\}.
 $$
 Furthermore, we have
 $$
 \sup_{x}|F(x)-\Phi(x)|\le C\sum_{i=1}^n\left\{\ee X_i^2I(|X_i|>1)+
 \ee |X_i|^3I(|X_i|\le 1)\right\}.
 $$
\end{lem}
\begin{lem}\label{lem5} $($\cite[Lemma 2]{C-R}$)$
For any random variables $X$, $Y$, real $x$ and constant $a>0$,
$$
\sup_{x}\big|\pp(X+Y\leq x)-\Phi(x)\big|\leq\sup_{x}\big|\pp(X\leq x)-\Phi(x)\big|+\frac{a}{\sqrt{2\pi}}+\pp\left(|Y|>a\right),
$$
where $\Phi(x)$ is the standard normal distribution.
\end{lem}
\begin{lem}\label{H}\cite{Hoeffding} Let $X_1, X_2, \cdots, X_n$ be independent random variables with $\ee X_i=0$ and $a_i\le X_i\le b_i$ for any $1\le i\le n$, where $a_1, b_1, a_2, b_2, \cdots, a_n, b_n$ are constants with $a_i<b_i$ for every $1\le i\le n$. Then for any $t>0$, we have
$$
\ee\exp\left(t\sum_{i=1}^n X_i\right) \le \exp\left(\frac{1}{8}t^2\sum_{i=1}^n(b_i-a_i)^2\right).
$$
In particular, for any $r>0$, we have
$$
\pp\left(\left|\sum_{i=1}^n X_i\right|>rn\right) \le 2\exp\left(-\frac{2 n^2r^2}{\sum_{i=1}^n(b_i-a_i)^2}\right).
$$
\end{lem}
\begin{proof} [{\bf Proof of Theorem \ref{thm2-1}}]
From the definition of the plug-in estimator $\hat H_n$ for the entropy $H_n$, we have
\beq\label{clt1-p1}
\aligned
\hat H_n- H_n=&-\sum_{i=1}^{K(n)}(\hat p_n(i)-p_n(i))\ln p_n(i) -\sum_{i=1}^{K(n)}\hat p_n(i)\ln \frac{\hat p_n(i)}{p_n(i)}.
\endaligned
\deq
For every $n\ge 1$ and $1\le k\le n$, let us define
\beq\label{clt1-p1t}
T_{k,n}:=-\sum_{i=1}^{K(n)}(I_{\{X_{k, n}=i\}}-p_n(i))\ln p_n(i),
\deq
then we have
\beq\label{clt1-p2}
\aligned
\sum_{i=1}^{K(n)}(\hat p_n(i)-p_n(i))\ln p_n(i)=\frac{1}{n}\sum_{k=1}^nT_{k,n}.
\endaligned
\deq
It is easy to check that $\{T_{k,n}, 1\le k\le n, n\ge 1\}$ is an array of independent identically distributed random variables with $\ee T_{1,n}=0$ and
\begin{align*}
Var(T_{1,n})=&\ee\left(\sum_{i=1}^{K(n)}(1_{\{X_{1, n}=i\}}-p_n(i))\ln p_n(i)\right)^2\\
=&\sum_{i=1}^{K(n)}p_{n}(i)\ln^2 p_{n}(i)-\left(\sum_{i=1}^{K(n)}p_{n}(i)\ln p_{n}(i)\right)^2\\
=& Var(\ln p_n(X_{1,n}))=\sigma^2_n.
\end{align*}
From the inequality $(a+b)^2\le 2(a^2+b^2)$, we have
\begin{align*}
\frac{1}{n\sigma^2_n}&\sum_{k=1}^n\ee T_{k,n}^21_{\{|T_{k,n}|>\varepsilon\sqrt{n}\sigma_n\}}
=\frac{1}{\sigma^2_n}\ee T_{1,n}^21_{\{|T_{1,n}|>\varepsilon\sqrt{n}\sigma_n\}}\\
\le& \frac{2}{\sigma^2_n}\ee\left(\sum_{i=1}^{K(n)}1_{\{X_{1, n}=i\}}\ln p_n(i)\right)^2 1_{\{|T_{1,n}|>\varepsilon\sqrt{n}\sigma_n\}}
 +\frac{2H_n^2}{\sigma^2_n}\pp\left(|T_{1,n}|>\varepsilon\sqrt{n}\sigma_n\right).
\end{align*}
Now we rewrite the event $\left\{|T_{1,n}|>\varepsilon\sqrt{n}\sigma_n\right\}$ as follows
\begin{align*}
\left\{|T_{1,n}|>\varepsilon\sqrt{n}\sigma_n\right\}=&\left\{\left|\sum_{i=1}^{K(n)}I_{\{X_{1, n}=i\}}\ln p_n(i)+H_n\right|>\varepsilon\sqrt{n}\sigma_n\right\}\\
=&\left\{\sum_{i=1}^{K(n)}I_{\{X_{1, n}=i\}}\ln p_n(i)>-H_n+\varepsilon\sqrt{n}\sigma_n\right\}\\
&\ \ \ \bigcup\left\{\sum_{i=1}^{K(n)}I_{\{X_{1, n}=i\}}\ln p_n(i)<-H_n-\varepsilon\sqrt{n}\sigma_n\right\}.
\end{align*}
By using Jensen's inequality, we have
$$
H_n=-\sum_{i=1}^{K(n)}p_{n}(i)\ln p_{n}(i)\le \ln K(n),
$$
which, together with the condition (\ref{1}), implies
$$
-H_n+\sqrt{n}\sigma_n\to +\infty.
$$
Hence, it is easy to check that the number of terms in $\{p_n(i): p_n(i)>e^{-H_n+\varepsilon\sqrt{n}\sigma_n}\}$ is less than or equal to
$e^{H_n-\varepsilon\sqrt{n}\sigma_n}$, which, implies that for all $n$ large enough, $H_n-\varepsilon\sqrt{n}\sigma_n<0$,\ \
$\sharp\{p_n(i): p_n(i)>e^{-H_n+\varepsilon\sqrt{n}\sigma_n}\}=0$ and
$$
\sum_{\{p_n(i): p_n(i)>e^{-H_n+\varepsilon\sqrt{n}\sigma_n}\}}p_n(i)(\ln p_n(i))^2=0.
$$
Moreover, we have
\begin{align*}
\frac{1}{\sigma^2_n}&\sum_{\{p_n(i): p_n(i)<e^{-H_n-\varepsilon\sqrt{n}\sigma_n}\}}p_n(i)(\ln p_n(i))^2\\
\le &\frac{K(n)}{\sigma^2_n}e^{-H_n-\varepsilon\sqrt{n}\sigma_n}(H_n+\varepsilon\sqrt{n}\sigma_n)^2\\
\le
&C\frac{K(n)}{\sigma^2_n}e^{-\varepsilon\sqrt{n}\sigma_n}(\varepsilon\sqrt{n}\sigma_n)^2
\le CK(n)ne^{-\sqrt{n}\sigma_n}\\
=&C\exp\left\{\ln n+\ln K(n)-\sqrt{n}\sigma_n\right\}
\to 0.
\end{align*}
Hence we have
\begin{align*}
& \frac{2}{\sigma^2_n}\ee\left(\sum_{i=1}^{K(n)}1_{\{X_{1, n}=i\}}\ln p_n(i)\right)^2 1_{\{|T_{1,n}|>\varepsilon\sqrt{n}\sigma_n\}}\\
= & \frac{2}{\sigma^2_n}\sum_{\{p_n(i):|\ln p_n(i)+H_n|>\varepsilon\sqrt{n}\sigma_n\}}p_n(i)(\ln p_n(i))^2\\
\le & \frac{2}{\sigma^2_n}\sum_{\{p_n(i): p_n(i)>e^{-H_n+\varepsilon\sqrt{n}\sigma_n}\}}p_n(i)(\ln p_n(i))^2\\
&\ \ \ \ \ +\frac{2}{\sigma^2_n}\sum_{\{p_n(i): p_n(i)<e^{-H_n-\varepsilon\sqrt{n}\sigma_n}\}}p_n(i)(\ln p_n(i))^2\\
\to & 0.
\end{align*}
Furthermore, we get
$$
\frac{2H_n^2}{\sigma^2_n}\pp\left(|T_{1,n}|>\varepsilon\sqrt{n}\sigma_n\right)\le \frac{2H_n^2}{\varepsilon^2n\sigma_n^2}\le \frac{2(\ln K(n))^2}{\varepsilon^2n\sigma_n^2}\to 0.
$$
So, by using Lindeberg central limit theorem, we have
$$
\frac{\sqrt{n}}{\sigma_n}\sum_{i=1}^{K(n)}(\hat p_n(i)-p_n(i))\ln p_n(i)\xrightarrow{\dd}N(0,1).
$$
Next it is enough to prove
\beq\label{thm1-101}
\frac{\sqrt{n}}{\sigma_n}\sum_{i=1}^{K(n)}\hat p_n(i)\ln \frac{\hat p_n(i)}{p_n(i)}\xrightarrow{\pp}0.
\deq
From the following inequality
$$
\frac{h}{1+h}<\ln (1+h)<h \ \ \ \text{for}\ \ h>-1,
$$
we have
\beq\label{thm1-102}
\sum_{i=1}^{K(n)}\hat p_n(i)\ln \frac{\hat p_n(i)}{p_n(i)}\le \sum_{i=1}^{K(n)}\hat p_n(i)\left(\frac{\hat p_n(i)}{p_n(i)}-1\right)=\sum_{i=1}^{K(n)}\frac{\left(\hat p_n(i)-p_n(i)\right)^2}{p_n(i)}
\deq
and
\beq\label{thm1-103}
\aligned
\sum_{i=1}^{K(n)}\hat p_n(i)\ln \frac{\hat p_n(i)}{p_n(i)}\ge \sum_{i=1}^{K(n)}\frac{\hat p_n(i)\left(\frac{\hat p_n(i)}{p_n(i)}-1\right)}{1+\left(\frac{\hat p_n(i)}{p_n(i)}-1\right)}
=\sum_{i=1}^{K(n)} p_n(i)\left(\frac{\hat p_n(i)}{p_n(i)}-1\right)=0.
\endaligned
\deq
So we can get
\begin{align*}
&\frac{\sqrt{n}}{\sigma_n}\sum_{i=1}^{K(n)}\ee \left[\frac{\left(\hat p_n(i)-p_n(i)\right)^2}{p_n(i)}\right]\\
=&\frac{1}{n^{3/2}\sigma_n}\sum_{i=1}^{K(n)} \frac{1}{p_n(i)}\ee\left(\sum_{j=1}^n(I_{\{X_{j,n}=i\}} -p_n(i))\right)^2\\
=& \frac{1}{\sqrt{n}\sigma_n}\sum_{i=1}^{K(n)}(1-p_n(i))= \frac{K(n)-1}{\sqrt{n}\sigma_n}\to 0,
\end{align*}
which implies
\beq\label{thm1-11}
\frac{\sqrt{n}}{\sigma_n}\sum_{i=1}^{K(n)}\hat p_n(i)\left(\frac{\hat p_n(i)}{p_n(i)}-1\right)\xrightarrow{\pp}0.
\deq
From (\ref{thm1-102}),  (\ref{thm1-103}) and (\ref{thm1-11}), we get the claim (\ref{thm1-101}).
\end{proof}
\begin{proof} [{\bf Proof of Theorem \ref{thm2-2}}]
From (\ref{clt1-p1}), (\ref{clt1-p2}) and Lemma \ref{lem5}, for any $a>0$, we have
\begin{align}\label{2201}
&\sup\limits_x\left|\pp\left(\frac{\sqrt{n}}{\sigma_n}(\hat{H}_n-H)<x\right)-\Phi(x)\right|\nonumber\\
\le &\sup\limits_x\left|\pp\left(\frac{1}{\sqrt{n}\sigma_n}\sum_{k=1}^nT_{k,n}<x\right)-\Phi(x)\right|
+\frac{a}{\sqrt{2\pi}}\nonumber\\
&\ \ \ +\pp\left(\frac{\sqrt{n}}{\sigma_n}\left|\sum_{i=1}^{K(n)}\hat p_n(i)\ln {\frac{\hat p_n(i)}{p_n(i)}}\right|>a\right).
\end{align}
Firstly, by Lemma \ref{lem1}, for any $0\le \delta\le 1$, we have
\begin{align}\label{2202}
&\sup\limits_x\left|\pp\left(\frac{1}{\sqrt{n}\sigma_n}\sum_{k=1}^nT_{k,n}<x\right)-\Phi(x)\right|\nonumber\\
\le &C\sum_{k=1}^n\left\{\ee \left(\frac{T_{k,n}}{\sqrt{n}\sigma_n}\right)^2I(|T_{k,n}|>\sqrt{n}\sigma_n)+
 \ee\left(\frac{T_{k,n}}{\sqrt{n}\sigma_n}\right)^3I(|T_{k,n}|\le\sqrt{n}\sigma_n)\right\}\nonumber\\
 \le & C\sum_{k=1}^n\left\{\ee \left|\frac{T_{k,n}}{\sqrt{n}\sigma_n}\right|^{2+\delta}I(|T_{k,n}|>\sqrt{n}\sigma_n)+
 \ee\left|\frac{T_{k,n}}{\sqrt{n}\sigma_n}\right|^{2+\delta}I(|T_{k,n}|\le\sqrt{n}\sigma_n)\right\}\nonumber\\
 \le & \frac{C}{n^{\delta/2}\sigma_n^{2+\delta}}\ee|T_{1,n}|^{2+\delta}.
\end{align}
Obviously, for any $n$ and any $0\le \delta\le 1$, we have
\beq\label{2}
\ee|T_{1,n}|^{2+\delta}=\ee\left|\sum_{i=1}^{K(n)}(I_{\{X_{1, n}=i\}}-p_n(i))\ln p_n(i)\right|^{2+\delta}<\infty.
\deq
From (\ref{thm1-102}) and (\ref{thm1-103}), we have
\begin{align}\label{2203}
&\pp\left(\frac{\sqrt{n}}{\sigma_n}\left|\sum_{i=1}^{K(n)}\hat p_n(i)\ln {\frac{\hat p_n(i)}{p_n(i)}}\right|>a\right)\nonumber\\
\le & \pp\left(\frac{\sqrt{n}}{\sigma_n}\sum_{i=1}^{K(n)}\frac{\left(\hat p_n(i)-p_n(i)\right)^2}{p_n(i)}>a\right)\le  \frac{K(n)-1}{a\sqrt{n}\sigma_n}.
\end{align}
By taking $a=\sqrt{\sqrt{2\pi}\frac{K(n)}{\sqrt{n}\sigma_n}}$,  (\ref{3}) holds by using (\ref{2201}), (\ref{2202}), (\ref{2}) and (\ref{2203}).
\end{proof}
\begin{proof}[{\bf Proof of Theorem \ref{thm2-3}}]
From (\ref{clt1-p1}), we have
$$
\hat H_n- H_n=-\sum_{i=1}^{K(n)}(\hat p_n(i)-p_n(i))\ln p_n(i) -\sum_{i=1}^{K(n)}\hat p_n(i)\ln \frac{\hat p_n(i)}{p_n(i)}.
$$
By using the definition in (\ref{clt1-p1t}), we have
\beq\label{mod1-p2}
\aligned
\frac{\sqrt{n}}{b_n\sigma_n}\sum_{i=1}^{K(n)}(\hat p_n(i)-p_n(i))\ln p_n(i)=\frac{1}{b_n\sqrt{n}\sigma_n}\sum_{k=1}^nT_{k,n}.
\endaligned
\deq
From the assumption of
$$
\sup_{n\ge 1}\ee\exp\left(\frac{\delta}{\sigma_n}|T_{1, n}|\right)<\infty,
$$
 then for any $\lambda\in\rr$, we have
\begin{align*}
&\lim_{n\rightarrow\infty}\frac{1}{b_n^2}\log\ee
\exp\left\{\lambda\frac{b_n}{\sqrt{n}\sigma_n}\sum_{k=1}^nT_{k,n}\right\}\\
=&\lim_{n\rightarrow\infty}\frac{n}{b_n^2}\log\ee
\exp\left\{\lambda\frac{b_n}{\sqrt{n}\sigma_n}T_{1,n}\right\}\\
=&\lim_{n\rightarrow\infty}\frac{n}{b_n^2}\log\ee\left(1+\frac{ \lambda b_n}{\sqrt{n}\sigma_n} T_{1,n}+\frac{ \lambda^2 b_n^2}{2n\sigma_n^2} T_{1,n}^2+o\left(\frac{b_n^2}{n}\right)\right)\\
=&\lim_{n\rightarrow\infty}\frac{n}{b_n^2}\log\left(1+\frac{ \lambda^2 b_n^2}{2n}+o\left(\frac{b_n^2}{n}\right)\right)\\
=&\frac{\lambda^2}{2}.
\end{align*}
By using G$\ddot{a}$rtner-Ellis Theorem (see \cite{2}), we have
\beq\label{mod1-p21}
\aligned
\lim_{n\rightarrow\infty}\frac{1}{b_n^2}\log\pp\left(\frac{\sqrt{n}}{b_n\sigma_n}\left|\sum_{i=1}^{K(n)}(\hat p_n(i)-p_n(i))\ln p_n(i)\right|>r\right)=-\frac{r^2}{2}.
\endaligned
\deq
Next, in order to prove Theorem \ref{thm2-3}, it is enough to show the following claim: for any $\varepsilon > 0$,
\beq\label{mod1-p22}
\lim_{n\rightarrow\infty}\frac{1}{b_n^2}\log\pp
\left(\frac{\sqrt{n}}{b_n\sigma_n}\left|\sum_{i=1}^{K(n)}\hat p_n(i)\ln \frac{\hat p_n(i)}{p_n(i)}\right|>\varepsilon\right)= -\infty.
\deq
From (\ref{thm1-102}), (\ref{thm1-103}) and Lemma \ref{H}, it is easy to check that
\begin{align*}
&\pp\left(\frac{\sqrt{n}}{b_n\sigma_n}\left|\sum_{i=1}^{K(n)}\hat p_n(i)\ln \frac{\hat p_n(i)}{p_n(i)}\right|>\varepsilon\right)\\
\le &\pp\left(\frac{\sqrt{n}}{b_n\sigma_n}\left|\sum_{i=1}^{K(n)}\frac{\left(\hat p_n(i)-p_n(i)\right)^2}{p_n(i)}\right|>\varepsilon\right)\\
\le & \sum_{i=1}^{K(n)}\pp\left(\frac{\sqrt{n}}{b_n\sigma_n}\left|\frac{\left(\hat p_n(i)-p_n(i)\right)^2}{p_n(i)}\right|>\varepsilon p_n(i)\right)\\
=&\sum_{i=1}^{K(n)}\pp\left(\left|\frac{1}{n}\sum_{j=1}^n(I_{\{X_j=i\}}-p_n(i))\right|>\frac{\sqrt{\varepsilon b_n\sigma_n}p_n(i)}{n^{1/4}}\right)\\
\le & 2\sum_{i=1}^{K(n)}\exp\left(-2\varepsilon\sqrt{n}b_n\sigma_np_n^2(i)\right),
\end{align*}
which implies (\ref{mod1-p22}).
\end{proof}

\end{document}